\documentclass [12pt]{article}

\usepackage{amsmath}
\usepackage{mathtools}
\usepackage{amssymb}
\usepackage{here}
\usepackage{fullpage}
\usepackage{mathrsfs}
\usepackage{enumerate}
\usepackage{multirow}
\usepackage{color}
\usepackage{cite}
\usepackage{dsfont}

\definecolor{lblue}{RGB}{0,110,152}

\definecolor{dred}{RGB}{171,67,53}

\newtheorem{theorem}{Theorem}%[\arabic{section}]

\newtheorem{proposition}[theorem]{Proposition}
\newtheorem{lemma}[theorem]{Lemma}
\newtheorem{define}[theorem]{Definition}
\newtheorem{example}[theorem]{Example}

\newtheorem{remark}[theorem]{Remark}

\newcommand{\mendth}{\hfill \ensuremath{\vartriangle}}

\DeclareMathOperator*{\col}{col}

\DeclareMathOperator*{\diag}{diag}
\DeclareMathOperator*{\rank}{rank}

\DeclareMathOperator{\sgn}{sgn}

\DeclareMathOperator{\E}{ \mathbb{E}}

\def\E{\mathbb{E}}

\providecommand\X[1]{\boldsymbol{X_{#1}}}
\providecommand\Z[1]{\boldsymbol{Z_{#1}}}

\providecommand\phib{\boldsymbol{\emptyset}}
\providecommand\Deltab{\boldsymbol{\Delta}}

\providecommand{\blue}[1]{\color{black}{#1}\color{black}\hspace{0pt}}

\providecommand{\dgreen}[1]{\color{black}{#1}\color{black}\hspace{0pt}}

\def\llongrightarrow{\relbar\joinrel\relbar\joinrel\relbar\joinrel\rightarrow}

\providecommand{\rarrow}[1]{\stackrel{#1}{\llongrightarrow}}

\def\Xz{\boldsymbol{X}}
\def\Zz{\boldsymbol{Z}}

\providecommand\X[1]{\boldsymbol{X_{#1}}}
\providecommand\Z[1]{\boldsymbol{Z_{#1}}}
\providecommand\phib{\boldsymbol{\emptyset}}

\newenvironment{proof}{{\it Proof :~}}{\hfill$\diamondsuit$\\}

\begin{document}

\title{Ergodicity, Output-Controllability, and Antithetic Integral Control of Uncertain Stochastic Reaction Networks}

\author{Corentin Briat and Mustafa Khammash\thanks{email: corentin@briat.info, mustafa.khammash@bsse.ethz.ch; url: www.briat.info, https://www.bsse.ethz.ch/ctsb.}\\D-BSSE, ETH-Z\"urich, Switzerland}

\date{}

\maketitle

\begin{abstract}
The ergodicity and the output-controllability of stochastic reaction networks have been shown to be essential properties to fulfill to enable their control using, for instance, antithetic integral control \cite{Briat:15e}. We propose here to extend those properties to the case of uncertain networks. To this aim, the notions of interval,  robust,  sign,  and structural ergodicity/output-controllability are introduced. The obtained results lie in the same spirit as those obtained in \cite{Briat:15e} where those properties are characterized in terms of control theoretic concepts, linear algebraic conditions, linear programs, and graph-theoretic/algebraic conditions. An important conclusion is that all those properties can be characterized by linear programs. Two examples are given for illustration.\\

\noindent Keywords. \emph{Stochastic reaction networks; Ergodicity; Output-Controllability; Antithetic Integral Control; Robustness; Cybergenetics}
\end{abstract}

\section{Introduction}

\blue{Reaction networks is a powerful formalism that can represent a wide variety of real world processes \cite{Goutsias:13}. When the dynamics of those processes is subject to randomness, as in biology \cite{Thattai:01,Elowitz:02}, stochastic reaction networks have been proven to play an essential role for their modeling, analysis, filtering, and control; see e.g. \cite{Anderson:15,Briat:13i,Zechner:16,Briat:15e}. Indeed, it is now well-known that stochastic reaction networks can exhibit several interesting properties that are absent for their deterministic counterparts \cite{Vilar:02,Paulsson:00,Briat:15e,Gupta:16}. Under the well-mixed assumption, the dynamics of a stochastic reaction network can be described by a continuous-time jump Markov process evolving on an integer lattice \cite{Anderson:15}. Sufficient conditions for checking the ergodicity of open unimolecular and bimolecular stochastic reaction networks have been proposed in \cite{Briat:13i} and formulated in terms of linear programs. The concept of ergodicity is of fundamental importance as it can serve as a basis for the development of a control theory for stochastic reaction networks and, consequently, of a control theory for biological systems. Indeed, verifying the ergodicity of a control system, consisting for instance of an endogenous biomolecular network controlled by a synthetic controller, would establish that the closed-loop network is well-behaved (e.g. globally converging first- and second-order moments) and that the designed control system achieves its goal (e.g. set-point tracking and perfect adaptation properties). This  procedure is analogous to that of checking the global stability of a closed-loop system in the deterministic setting; see e.g. \cite{DelVecchio:16}. Additionally, designing synthetic circuits that are provably ergodic could allow for the rational design of synthetic networks that can exploit noise in their function. A recent example is that of the antithetic integral feedback controller proposed in \cite{Briat:15e} that has been shown to induce an ergodic closed-loop network when the open-loop network is both ergodic and output-controllable -- a closed-loop property that holds regardless the values of the controller parameters. However, a major limitation of the ergodicity and output-controllability conditions obtained in \cite{Briat:13i,Briat:15e} is their limited scope to networks with fixed and known rate parameters only -- an assumption rarely met in practice. This has motivated the consideration of stochastic reaction networks with uncertain rate parameters in\cite{Briat:16cdc,Briat:17ifacBio,Briat:17LAA}.

The objective of this paper is to provide a global picture of all the obtained results in \cite{Briat:13i,Briat:15e,Briat:16cdc,Briat:17ifacBio} by unifying and generalizing them, by providing comparisons between them, and by emphasizing their connections with results in systems theory, control theory, linear algebra, and optimization. This unified picture is obtained through the introduction of the concepts of interval-, robust-, sign-, and structural-ergodicity (resp. output controllability) for uncertain stochastic reaction networks (resp. for uncertain linear positive systems). Unlike in \cite{Briat:16cdc,Briat:17ifacBio}, all the results are stated with their proof. Novel results are also provided as a way to consolidate the structure of this unifying viewpoint.

The interval-approach considers classes of networks described by interval matrices \cite{Moore:09}. We show that checking the ergodicity and the output controllability of the entire network family reduces to checking the ergodicity of a single network and the output controllability of a single linear positive system, two problems which naturally reformulate as simple linear programs. Unlike the interval-approach, the robust approach considers the explicit dependence on the rate parameters of the matrices describing the network. In this regard, this approach may be conclusive whenever the interval-approach fails, a scenario plausible to happen when the considered network involves conversion reactions. The price to pay, however, is a higher computational cost for establishing the robust ergodicity property. Checking the output-controllability property remains the same as in the interval approach.

The sign-approach, more qualitative in nature, is based on sign-matrices \cite{Jeffries:77,Brualdi:95,Briat:17LAA,McCreesh:18} which have been extensively studied and considered for the qualitative analysis of dynamical systems, including  reaction networks, albeit much more sporadically; see e.g. \cite{Helton:10,Briat:16cdc,Giordano:16,Briat:17LAA}. In this case, again, the ergodicity and output controllability conditions can be stated as simple linear programs; see e.g. \cite{Briat:16cdc,Briat:17LAA}. The computational complexity is, hence, the same as in the interval approach. Finally, the structural-approach considers the exact parameter dependence, as the robust one. Ergodicity and output-controllability conditions are also formulated as simple linear programs under some realistic assumptions. When those conditions are not met, more expensive solutions can be obtained in the same flavor as in the robust case.}

\noindent\textbf{Outline.} We recall in Section \ref{sec:prel} several definitions and results related to reaction networks and antithetic integral control. Those concepts are extended to uncertain networks in Section \ref{sec:extensions}. Section \ref{sec:interval}, Section \ref{sec:robust}, Section \ref{sec:sign}, and Section \ref{sec:structural} extend the results of Section \ref{sec:prel} to the interval, robust, sign, and structural cases, respectively. Examples are treated in Section \ref{sec:ex}.

\noindent\textbf{Notations.} The standard basis for $\mathbb{R}^d$ is denoted by $\{e_i\}_{i=1}^d$. The sets of integers, nonnegative integers, nonnegative real numbers and  positive real numbers are denoted by $\mathbb{Z}$, $\mathbb{Z}_{\ge0}$, $\mathbb{R}_{\ge0}$ and $\mathbb{R}_{>0}$, respectively. The $d$-dimensional vector of ones is denoted by $\mathds{1}_d$ (the index will be dropped when the dimension is obvious). For vectors and matrices, the inequality signs $\le$ and $<$ act componentwise. Finally, the vector or matrix obtained by stacking the elements $x_1,\ldots,x_d$ is denoted by $\col_{i=1}^d(x_i)$ or $\col(x_1,\ldots,x_d)$. The diagonal operator $\diag(\cdot)$ is defined analogously. The spectral radius of a matrix $M\in\mathbb{R}^{n\times n}$ is defined as $\varrho(M)=\max\{|\lambda|:\det(\lambda I-M)=0\}$.

\section{Preliminaries}\label{sec:prel}

\blue{\subsection{SISO linear positive systems}\label{sec:LPS}

SISO linear systems are systems of the form
\begin{equation}\label{eq:LPS}
  \begin{array}{rcl}
    \dot{x}(t)&=&Mx(t)+bu(t)\\
    y(t)&=&c^Tx(t)\\
    x(0)&=&x_0
  \end{array}
\end{equation}
where $x,x_0\in\mathbb{R}^d$, $u\in\mathbb{R}$ and $y\in\mathbb{R}$ are the state of the system, the initial condition, the input and the output of the system. We also have that $M\in\mathbb{R}^{d\times d}$ and $b,c\in\mathbb{R}^d$. The above system is said to be internally positive if for any nonnegative initial condition and any nonnegative input, the state and the output are nonnegative. A necessary and sufficient condition for the internal positivity of \eqref{eq:LPS} is that $M$ is Metzler\footnote{A square matrix is Metzler if all its off-diagonal elements are nonnegative.} and $b,c$ are nonnegative.

We have the following result \cite{Haddad:05}:
\begin{proposition}\label{prop:stability}
  Assume that the system \eqref{eq:LPS} is internally positive. Then, the following statements are equivalent:
  \begin{enumerate}[(a)]
    \item The system \eqref{eq:LPS} with $u\equiv0$ is asymptotically stable;
    \item The matrix $M$ is Hurwitz stable;
    \item There exists a vector $v\in\mathbb{R}^d_{>0}$ such that $v^TM<0$; and
    \item $M$ is nonsingular and $M^{-1}\le0$.
  \end{enumerate}
\end{proposition}}

\blue{We will also need the following result on the output controllability of linear SISO positive systems which is an extension of the results in \cite{Briat:15e,Briat:16cdc,Briat:17ifacBio} to the case of non-necessarily stable systems:
\begin{proposition}\label{prop:outputcontrollability}
  Assume that the system \eqref{eq:LPS} is internally positive. Then, the following statements are equivalent:
  \begin{enumerate}[(a)]
  \item The system $(M,b,c^T)$ defined in \eqref{eq:LPS}  is output controllable.
    \item $\rank\begin{bmatrix}
      c^Tb & c^TMb & \ldots & c^TM^{d-1}b
    \end{bmatrix}=1$.
  \item There exists a vector $w\in\mathbb{R}^d_{\ge0}$ and a scalar $\mu\in\mathbb{R}_{\ge0}$ such that $w^Tb>0$ and $w^T(M-\mu I_d)+c^T=0$.
    \item There exists a scalar $\mu\ge0$ such that $-c^T(M-\mu I_d)^{-1}b>0$ holds or, equivalently, the static-gain of the system $(M-\mu I,b,c^T)$ is positive.
    \end{enumerate}
    When the $b=e_i$ and $c=e_j$, the above statements are equivalent to
\begin{enumerate}[(a)]
  \setcounter{enumi}{4}
    \item There is a path from node $i$ to node $j$ in the directed graph $\mathcal{G}_M=(\mathcal{V},\mathcal{E})$ defined with $\mathcal{V}:=\{1,\ldots,d\}$ and
\begin{equation*}
    \mathcal{E}:=\{(m,n):\ e_{n}^TMe_{m}\ne0,\ m,n\in V,\ m\ne n\}.\hfill\mendth
\end{equation*}
%  \end{enumerate}
%  When the matrix $A$ is Hurwitz stable, then the above statements are also equivalent to:
%\begin{enumerate}[(a)]
%  \setcounter{enumi}{4}
\end{enumerate}
Moreover, when the matrix $M$ is Hurwitz stable, then, the statements (c) and (d) hold with $\mu=0$.\mendth
\end{proposition}}
\blue{\begin{proof}
  The equivalence between the statements (a) and (b) follows from the definition. The equivalence with (e) can be found in \cite{Briat:15e}. The equivalence between (c) and (d) follows from choosing $\mu\ge0$ such that $M-\mu I$ is Hurwitz stable and $w^T=-c^T(M-\mu I)^{-1}\ge0$ where the nonnegativity follows from the fact that the matrix $M-\mu I$ is Metzler and Hurwitz stable. Finally, the equivalence between (a) and (c) follows from the fact that the system $(M,b,c^T)$ is output-controllable if and only if $(M-\mu I,b,c^T)$ is output-controllable.
\end{proof}}

\subsection{General stochastic reaction networks with mass-action kinetics}\label{sec:RN}

\dgreen{We consider here a reaction network with $d$ molecular species $\X{1},\ldots,\X{d}$ that interacts through $K$ reaction channels $\mathcal{R}_1,\ldots,\mathcal{R}_K$ defined as
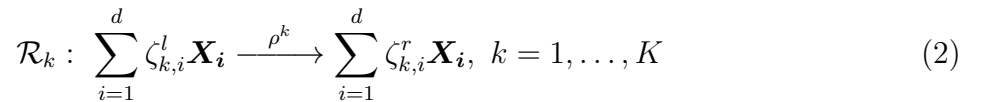
\begin{equation}
 \mathcal{R}_k:\ \sum_{i=1}^d\zeta_{k,i}^l\X{i}\rarrow{\rho^k}  \sum_{i=1}^d\zeta_{k,i}^r\X{i},\ k=1,\ldots,K
\end{equation}
where $\rho^k\in\mathbb{R}_{>0}$ is the reaction rate parameter and $\zeta_{k,i}^l,\zeta_{k,i}^r\in\mathbb{Z}_{\ge0}$. Each reaction is additionally described by a stoichiometric vector and a propensity function. \blue{Each reaction rate parameter is distinct and independent from the others.} The stoichiometric vector of reaction  $\mathcal{R}_k$ is given by $\zeta_k:=\zeta_k^r-\zeta_k^l\in\mathbb{Z}^d$ where $\zeta_k^r=\col(\zeta_{k,1}^r,\ldots,\zeta_{k,d}^r)$ and $\zeta_k^l=\col(\zeta_{k,1}^l,\ldots,\zeta_{k,d}^l)$. In this regard, when the reaction $\mathcal{R}_k$ fires, the state jumps from $x$ to $x+\zeta_k$. We define the stoichiometry matrix $S\in\mathbb{Z}^{d\times K}$ as $S:=\begin{bmatrix}
  \zeta_1\ldots\zeta_K
\end{bmatrix}$. When  the kinetics is  mass-action, the propensity function of reaction $\mathcal{R}_k$ is given by  $\textstyle\lambda_k(x)=\rho^k\prod_{i=1}^d\frac{x_i!}{(x_i-\zeta_{k,i}^l)!}$ and is such that  $\lambda_k(x)=0$ if $x\in\mathbb{Z}_{\ge0}^d$ and $x+\zeta_k\notin\mathbb{Z}_{\ge0}^d$. We denote this reaction network by $(\Xz,\mathcal{R})$. Under the well-mixed assumption, this network can be described by a continuous-time Markov process $(X_1(t),\ldots,X_d(t))_{t\ge0}$ with state-space $\mathbb{Z}_{\ge0}^d$; see e.g. \cite{Anderson:11}.}
\blue{This Markov process is fully described by the Chemical Master Equation or Forward Kolmogorov Equation given by \cite{Anderson:11}
\begin{equation}
  \dfrac{\partial p_{x_0}(x,t)}{\partial t}=\sum_{i=1}^K\lambda_i(x)(p_{x_0}(x+\zeta_i,t)-p_{x_0}(x,t))
\end{equation}
where $p_{x_0}(x,t)$ is the probability for the Markov process to be in state $x$ at time $t$, starting from the initial state $X(0)=x_0$. Knowing the probability distribution provides a lot of information about the behavior of the Markov process and the associated reaction network. Unfortunately, this equation is difficult to solve in general and alternative ways to study its behavior need to be considered; see e.g. \cite{Meyn:93}. In particular, it is interesting to know whether there is a unique attractive stationary distribution. This leads to the following definition:
\begin{define}[Ergodicity of a reaction network]
  The Markov process associated with the reaction network $(\Xz,\mathcal{R})$ is said to be ergodic if its probability distribution  $p_{x_0}(x,\cdot)$ globally converges to a unique stationary distribution $\pi$; i.e. for every $x_0\in\mathbb{Z}_{\ge0}^d$, we have that $p_{x_0}(x,t)\to \pi$ as $t\to\infty$. The network is exponentially ergodic if the convergence to the stationary distribution is exponential.
\end{define}
This definition is the stochastic analogue of a globally attracting equilibrium point for deterministic dynamics.}

\blue{
\begin{define}[Irreducible reaction network]\label{def:irreducibility}
  A stochastic reaction network is said to be irreducible\footnote{Computationally tractable conditions for checking the irreducibility of reaction networks are provided in \cite{Gupta:18}.} if the state-space of the underlying Markov process is irreducible.
\end{define}

\begin{define}[Open reaction network]\label{def:open}
  A reaction network is said to be open if there is no set of conserved species in the network; i.e. $z\in\mathbb{R}^d_{\ge0}$ and $z^TS=0 \Rightarrow z=0$)
\end{define}}

\subsection{Bimolecular stochastic reaction networks}

\dgreen{Let us assume here that the network  $(\Xz,\mathcal{R})$ is at most bimolecular and that the reaction rates are all independent of each other. In such a case, the propensity functions are polynomials of at most degree 2 and we can write the propensity vector as
\begin{equation}
  \lambda(x)=\begin{bmatrix}
w_0(\rho_0)\\
W(\rho_u)x\\
W_b(\rho_b,x)
  \end{bmatrix}
\end{equation}
where $w_0(\rho_0)\in\mathbb{R}^{n_0}_{\ge0}$, $W(\rho_u)x\in\mathbb{R}^{n_u}_{\ge0}$ and $W_b(\rho_b,x)\in\mathbb{R}^{n_b}_{\ge0}$ are the propensity vectors associated the zeroth-, first- and second-order reactions, respectively. Their respective rate parameters are also given by $\rho_0$, $\rho_u$ and $\rho_b$, and according to this structure, the stoichiometric matrix is decomposed as
 $S=:\begin{bmatrix}S_0 &
  S_u & S_b
\end{bmatrix}$. Before stating the main results of the section, we need to introduce the following terminology:}

\begin{define}\label{def:Ab}
  The \emph{characteristic matrix} $A(\rho_u)$ and the \emph{offset vector} $b_0(\rho)$ of a bimolecular reaction network $(\Xz,\mathcal{R})$ are defined as
  \begin{equation}
    A(\rho_u):=S_uW(\rho_u)\ \textnormal{and}\ b_0(\rho_0):=S_0w_0(\rho_0).
  \end{equation}
  Moreover, the matrix $A(\rho_u)$ is Metzler and the vector $b_0(\rho_0)$ is nonnegative for all positive rate parameters.
\end{define}

\blue{\begin{define}
  The dynamics of the first-order moments of a stochastic bimolecular reaction network $(\Xz,\mathcal{R})$ is described by the internally positive system
  \begin{equation}\label{eq:first-order}
      \dfrac{d\E[X(t)]}{dt}=A(\rho_u)\E[X(t)]+b_0(\rho_0)+S_b\E[W_b(\rho_b,X(t))]
  \end{equation}
  where $\E[X(0)]=x_0$.
\end{define}
Note that when the network is unimolecular, then the moments dynamics is described by a linear internally positive system; i.e. $A(\rho_u)$ is Metzler and $b_0(\rho_0)$ is nonnegative.
}

\subsection{Ergodicity of unimolecular and bimolecular reaction networks}

\blue{We have the following result which is a slight extension of a result in \cite{Briat:13i}:
\begin{theorem}[Ergodicity of unimolecular networks]\label{th:uni:erg}
    Let us consider an open irreducible unimolecular reaction network $(\Xz,\mathcal{R})$ with fixed rate parameters; i.e. $A=A(\rho_u)$ and $b_0=b_0(\rho_0)$. Then, the following statements are equivalent:
    \begin{enumerate}[(a)]
      \item The reaction network $(\Xz,\mathcal{R})$ is exponentially ergodic and all the moments are bounded and converging;
      \item There exists a vector $v\in\mathbb{R}^d_{>0}$ such that $v^TA<0$;
      \item The matrix $A$ is Hurwitz stable.
    \end{enumerate}
\end{theorem}
\begin{proof}
The equivalence between the two last statements follows from Proposition \ref{prop:outputcontrollability}. We also note that for ergodic unimolecular networks all moments globally converge to their unique fixed point; see \cite{Briat:13i}. Now assume that (a) holds. Using the fact that there is no conserved set of species, this implies that \eqref{eq:first-order} with $S_b=0$  globally converges to a unique equilibrium point. A necessary and sufficient condition for that is that $A$ be Hurwitz stable; i.e. (c) holds. To prove the converse, we assume that the network is non-ergodic. Since the state-space is irreducible then the network can only be non-ergodic if its trajectories grow unboundedly or, equivalently, the first-order moments diverge. This implies that $A$ must not be Hurwitz stable.  This proves the result.
\end{proof}
We then have the following generalization to bimolecular networks \cite{Briat:13i}:
\begin{theorem}[Ergodicity of bimolecular networks]\label{th:bi:erg}
  Let us consider an open irreducible bimolecular reaction network $(\Xz,\mathcal{R})$ with fixed rate parameters; i.e. $A=A(\rho_u)$ and $b_0=b_0(\rho_0)$. Assume that there exists a vector $v\in\mathbb{R}^d_{>0}$ such that $v^TS_b=0$ and $v^TA<0$. Then, the  reaction network $(\Xz,\mathcal{R})$ is exponentially ergodic and all the moments are bounded and converging.
\end{theorem}}

\subsection{Antithetic integral control of unimolecular networks}

\dgreen{Antithetic integral control has been first proposed in \cite{Briat:15e} for solving the perfect adaptation problem in stochastic reaction networks. The underlying idea is to augment the open-loop network $(\Xz,\mathcal{R})$ with an additional set of species and reactions (the controller). The usual set-up is that this controller network acts on the production rate of the molecular species $\X{1}$ (the \emph{actuated species}) in order to steer the mean value of the \emph{controlled species} $\X{\ell}$, $\ell\in\{1,\ldots,d\}$, to a desired set-point (the reference) and ensure perfect adaptation for the controlled species. As proved in \cite{Briat:15e}, the antithetic integral control motif $(\Zz,\mathcal{R}^c)$ defined with
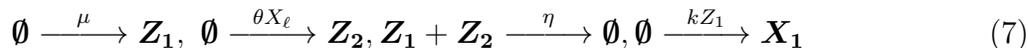
\begin{equation}\label{eq:AIC}
  \phib\rarrow{\mu}\Z{1},\ \phib\rarrow{\theta X_\ell}\Z{2}, \Z{1}+\Z{2}\rarrow{\eta}\phib, \phib\rarrow{kZ_1}\X{1}
\end{equation}
solves this control problem with the set-point being equal to $\mu/\theta$. Above, $\Z{1}$ and $\Z{2}$ are the \emph{controller species}. The four controller parameters $\mu,\theta,\eta,k>0$ are assumed to be freely assignable to any desired value. The first reaction is the \emph{reference reaction} as it encodes part of the reference value $\mu/\theta$ as its own rate. The second one is the \emph{measurement reaction} that produces the species $\Z{2}$ at a rate proportional to the current population of the controlled species $\X{\ell}$. The third reaction is the \emph{comparison reaction} as it compares the populations of the controller species and annihilates one molecule of each when these populations are both positive. Finally, the fourth reaction is the \emph{actuation reaction} that produces the actuated species $\X{1}$ at a rate proportional to the controller species $\Z{1}$. } \blue{The closed-loop reaction network consisting of the interconnection of the network $(\Xz,\mathcal{R})$ with the antithetic integral controller \eqref{eq:AIC} is denoted by $((\Xz,\Zz),\mathcal{R}\cup\mathcal{R}^c)$. Note also that the moments dynamics of the open loop network is therefore given by \eqref{eq:LPS} with $M=A(\rho_u)$, $b=e_1$ and $c=e_\ell$.}

We are now ready to state the main result of the section:
\begin{theorem}[\cite{Briat:15e}]\label{th:affine:nominal}
Let us consider an open unimolecular reaction network $(\Xz,\mathcal{R})$ with fixed characteristic matrix $A=A(\rho_u^0)$ and offset vector $b_0=b_0(\rho_0^0)$ for some nominal parameter values $\rho_0^0$ and $\rho_u^0$. Assume further that the closed-loop reaction network $((\Xz,\Zz),\mathcal{R}\cup\mathcal{R}^c)$ is irreducible. Then, the following statements are equivalent:
  \begin{enumerate}[(a)]
     \item The open-loop reaction network $(\Xz,\mathcal{R})$ is ergodic and the system  $(A,e_1,e_\ell^T)$ is output controllable.
    \item There exist vectors $v\in\mathbb{R}_{>0}^d$, $w\in\mathbb{R}_{\ge0}^d$, $w_1>0$, such that $v^TA <0$ and $w^TA +e_\ell^T=0$.
  \end{enumerate}
  Moreover, when one of the above statements holds, then the closed-loop reaction network $((\Xz,\Zz),\mathcal{R}\cup\mathcal{R}^c)$ is ergodic and we have that $\E[X_\ell(t)]\to\mu/\theta$ as $t\to\infty$  for any values for the controller rate parameters $\eta,k>0$ provided that
    \begin{equation}\label{eq:lowerbound}
      \dfrac{\mu}{\theta}>\dfrac{v^Tb_0}{\alpha e_\ell^Tv}
    \end{equation}
   where $\alpha>0$ and $v\in\mathbb{R}_{>0}^d$ verify $v^T(A+cI)\le0$.\mendth
\end{theorem}

\begin{remark}
  Interestingly, the conditions stated in the above result can be numerically verified by checking the feasibility of the following linear program
  \begin{equation}\label{feas:nom}
    \begin{array}{rcl}
      \textnormal{Find} && v\in\mathbb{R}_{>0}^d,\ w\in\mathbb{R}_{\ge0}^d\\
      \textnormal{s.t.}&&w^Te_1>0,v^TA<0,w^TA+e_\ell^T=0.
    \end{array}
  \end{equation}
\end{remark}

\section{Notions of ergodicity and output-controllability for uncertain unimolecular networks}\label{sec:extensions}

\blue{We address two main families of parameters. The first one is that of bounded parameter values
\begin{equation}
\mathcal{P}_u\subset\mathbb{R}_{\ge0}^{n_u}
\end{equation}
where $\mathcal{P}_u$'s is the cartesian product of connected intervals, which are not necessarily closed. The second type is that of unbounded parameter values
\begin{equation}
\mathcal{P}_u^\infty:=\mathbb{R}_{>0}^{n_u}.
\end{equation}
Depending on the considered type of parameters, different concepts can be defined. Those concepts are summarized in Table \ref{tab:frameworks}.}

\begin{table}[h]
\center
\blue{\caption{List of the considered different concepts with their exactness and domain of definition.}\label{tab:frameworks}
\begin{tabular}{|c||c|c|}
\hline
   & Bounded parameters & Unbounded parameters\\
   \hline
   \hline
   Approximated model & Interval concepts & Sign concepts\\
   \hline
   Exact model & Robust concepts & Structural concepts\\
   \hline
\end{tabular}}
\end{table}

\subsection{Interval ergodicity and output-controllability}

\blue{In this case, we assume that the characteristic matrix of the network belongs to the set
\begin{equation}\label{eq:calA}
  \mathcal{A}:=\left\{M\in\mathbb{R}^{d\times d}:\ A^-\le M\le A^+\right\},\ A^-\le A^+,
\end{equation}
where the matrices $A^-$ and $A^+$ verify $  A^-\le A(\rho_u)\le A^+$ holds for all $\rho_u\in\mathcal{P}_u$. In other words, we have that
\begin{equation}\label{eq:intervalmodel}
  \{A(\rho_u):\rho_u\in\mathcal{P}_u\}\subset \mathcal{A}.
\end{equation}
Alternatively, we can define the interval matrix $[A]$ such as $e_i^T[A]e_j=[a^-_{ij},a^+_{ij}]$; i.e. its $(i,j)$'s element is an interval.

We can then define the concepts of interval ergodicity and interval output-controllability:
\begin{define}[Interval ergodicity]
The unimolecular network $(\Xz,\mathcal{R})$ with interval characteristic matrix $[A]$  is interval (exponentially) ergodic if for each $A\in\mathcal{A}$, the network with characteristic matrix $A$ is (exponentially) ergodic.
\end{define}
\begin{define}[Interval output-controllability]
The linear interval system $([A],e_i,e_j^T)$ is interval output-controllable if for each $A\in\mathcal{A}$, the system $(A,e_i,e_j^T)$ is output-controllable.
\end{define}}

\subsection{Robust ergodicity and output-controllability}

\blue{The robust case considers the exact parameter dependence of the characteristic matrix. This leads to the following concepts of robust ergodicity and robust output-controllability:
\begin{define}[Robust ergodicity]
The unimolecular network $(\Xz,\mathcal{R})$ with parameter-dependent characteristic matrix $A(\rho_u)$, $\rho_u\in\mathcal{P}_u$,  is robustly (exponentially) ergodic if for each $\theta\in\mathcal{P}_u$, the network with characteristic matrix $A(\theta)$ is (exponentially) ergodic.
\end{define}
\begin{define}[Robust output-controllability]
The linear system $(A(\rho_u),e_i,e_j^T)$, $\rho_u\in\mathcal{P}_u$,  is robustly output-controllable if for each $\theta\in\mathcal{P}_u$, the system $(A(\theta),e_i,e_j^T)$ is output-controllable.
\end{define}}

\subsection{Sign ergodicity and output-controllability}

The sign ergodicity addresses the ergodicity of all the reaction networks having a characteristic matrix sharing the same given sign pattern. We define now the set of \emph{sign symbols} $\mathbb{S}:=\{0,\oplus,\ominus\}$. A \emph{sign-matrix} is a matrix with entries in $\mathbb{S}$ and the qualitative class $\mathcal{Q}(\Sigma)$ of a sign-matrix $\Sigma\in\mathbb{S}^{m\times n}$ is defined as
\begin{equation}
  \mathcal{Q}(\Sigma):=\left\{M\in\mathbb{R}^{m\times n}:\ \sgn(M)=\sgn(\Sigma)\right\}
\end{equation}
where the signum function $\sgn(\cdot)$ is defined as
\begin{equation}
    [\sgn(\Sigma)]_{ij}:=\left\{\begin{array}{ll}
                            1&\textnormal{if}\ \Sigma_{ij}\in\mathbb{R}_{>0}\cup\{\oplus\},\\
                            -1&\textnormal{if}\ \Sigma_{ij}\in\mathbb{R}_{<0}\cup\{\ominus\},\\
                            0&\textnormal{otherwise}.
                            \end{array}\right.
\end{equation}
\blue{Starting from $A(\rho_u)$, we can build the associated sign-pattern as $S_A=A(\oplus)$ where $A(\oplus)$ stands for the matrix where we have replaced all the parameters by $\oplus$ and used the arithmetic rules $-\oplus=\ominus$ and $\oplus+\oplus=\oplus$. Under the assumption that the network does not involve any conversion reactions, we have that
\begin{equation}
  \{A(\rho_u):\rho_u\in\mathcal{P}_u^\infty\}\subset \mathcal{Q}(S_A).
\end{equation}
 This leads to the following concepts of sign ergodicity and sign output-controllability:
\begin{define}
Assume that the unimolecular network $(\Xz,\mathcal{R})$ with characteristic sign-matrix $S_A$ does not contain any conversion reaction. Then, it is sign (exponentially) ergodic if for each $A\in\mathcal{Q}(S_A)$,  the network with characteristic matrix $A$ is (exponentially) ergodic.
\end{define}
\begin{define}
The linear sign system $(S_A,e_i,e_j^T)$ is sign output controllable if for each $A\in\mathcal{Q}(S_A)$, the system $(A,e_i,e_j^T)$ is output-controllable.
\end{define}}

\subsection{Structural ergodicity and output-controllability}

\blue{The structural case considers the exact parameter dependence of the characteristic matrix. This leads to the following concepts of structural ergodicity and structural output-controllability:
\begin{define}
The unimolecular network $(\Xz,\mathcal{R})$ with parameter-dependent characteristic matrix $A(\rho_u)$, $\rho_u\in\mathcal{P}_u^\infty$, is structurally (exponentially) ergodic if for each $\theta\in\mathcal{P}_u^\infty$, the network with characteristic matrix $A(\theta)$ is (exponentially) ergodic.
\end{define}
\begin{define}
The linear parameter-dependent system $(A(\rho_u),e_i,e_j^T)$, $\rho_u\in\mathcal{P}_u^\infty$, is structurally output-controllable if for each $\theta\in\mathcal{P}_u^\infty$, the system $(A(\theta),e_i,e_j^T)$ is output-controllable.
\end{define}}

\subsection{Equivalence between the concepts}

\blue{It seems interesting to start with some illustrative examples. The first one illustrates the impact of conversion reactions on the tightness of the interval and sign approximation.
\begin{example}[Hurwitz stability]
Let us consider the matrix
\begin{equation}
  A(\rho_u)=\begin{bmatrix}
    -\rho^1-\rho^2 & \rho^4\\
    \rho^2 & -\rho^3-\rho^4
  \end{bmatrix}, \rho_u=(\rho^1,\rho^2,\rho^3,\rho^4)
\end{equation}
where the rates $\rho^2$ and $\rho^4$ are rates of conversion reactions. This matrix is Hurwitz stable for all positive parameter values, hence it is structurally. It is also robustly stable provided the diagonal elements are bounded away from 0. However, if we pick $\rho^1=\rho^3=1$ and $\rho_2,\rho_4\in[1,3]$, then it  is not interval stable since
\begin{equation}
  A^+=\begin{bmatrix}
    -2  & 4\\
    4 & -2
  \end{bmatrix}
\end{equation}
is not Hurwitz stable. Similarly, the sign representation given by
\begin{equation}
 A(\oplus)=\begin{bmatrix}
    \ominus & \oplus\\
    \oplus & \ominus
  \end{bmatrix}
\end{equation}
is not sign-stable.
\end{example}

The second one addresses the case where the set of values of the parameters is not closed.
\begin{example}
  Let us consider the following matrix
  \begin{equation}
    A=\begin{bmatrix}
      -1 & 0\\
      \rho_1 & 0
    \end{bmatrix},b=e_1,c=e_2.
  \end{equation}
  Clearly, the system $(A,b,c)$ is structurally output-controllable. However, it is not interval observable if one considers that $\rho\in(0,1]$, since the matrix $A^-$ will have 0 as bottom-left entry.
  \end{example}

The above remarks are formalized below:
\begin{proposition}
  Assume that the network $(\Xz,\mathcal{R})$ is unimolecular. Then, the following statements are equivalent:
  \begin{enumerate}[(a)]
    \item There is no conversion reactions in the network $(\Xz,\mathcal{R})$ and the set $\mathcal{P}_u$ is closed.
    \item We have that $\mathcal{A}=\{A(\rho_u):\rho_u\in\mathcal{P}_u\}$.
  \end{enumerate}
\end{proposition}
\begin{proof}
  Clearly, if there are conversion reactions, then the sets do not match since the bounds $A^-,A^+$ will not be tight due to the presence of the same parameters in different entries; i.e both on the diagonal and in the corresponding columns. Indeed, in such a case, there is no $\rho_u\in\mathcal{P}_u$ such that $A(\rho_u)=A^-$ or $A(\rho_u)=A^-$. Additionally, if the set $\mathcal{P}_u$ is not closed, the bounds are never attained. This proves the result.
\end{proof}

The following result provides conditions under which the sign formulation is exact:
\begin{proposition}
  Assume that the network $(\Xz,\mathcal{R})$ is unimolecular.  The following statements are equivalent:
  \begin{enumerate}[(a)]
    \item There is no conversion reaction in the reaction network $(\Xz,\mathcal{R})$.
    \item  We have that $\{A(\rho_u):\rho_u\in\mathbb{R}_{>0}^{n_u}\}=\mathcal{Q}(S_A)$.
  \end{enumerate}
\end{proposition}
\begin{proof}
  If there is no conservation reaction, then all the entries in the matrix are independent and, therefore, the sign approach becomes non-conservative.
\end{proof}}

\subsection{Extensions to more general networks}

It is interesting to discuss whether those concepts extend to some bimolecular networks. The robust and structural definitions can be extended to any type of reaction networks as those definitions lie at the level of the reaction rates. The other definitions can be adapted to a class of bimolecular networks through the use of Theorem \ref{th:bi:erg}. Note that the sign-ergodicity of bimolecular networks has been addressed in \cite[Proposition 8.8]{Briat:17LAA} through the concept of $\textnormal{Ker}_+(B)$-sign-stability.

\section{Interval results}\label{sec:interval}

The objective of this section is to develop interval-analogues of the ergodicity and output-controllability results of Section \ref{sec:prel}.

\subsection{Interval ergodicity of unimolecular reaction networks}

\blue{Let us consider set of matrices \eqref{eq:calA} where the bounds $A^-,A^+$ have been determined such that \eqref{eq:intervalmodel} holds. Then, we have the following result:
\begin{proposition}\label{lem:frob}
Let us consider an irreducible open unimolecular reaction network $(\X{},\mathcal{R})$ with interval characteristic matrix $[A]$. Then, the following statements are equivalent:
\begin{enumerate}[(a)]
  \item The network $(\Xz,\mathcal{R})$ is interval (exponentially) ergodic;
  \item All the matrices in $\mathcal{A}$ are Hurwitz stable or, equivalently, for any $M\in\mathcal{A}$, there exists a $v=v(M)\in\mathbb{R}^d_{>0}$ such that $v^TM<0$;
  \item The matrix $A^+$ is Hurwitz stable or, equivalently, there exists a vector $v_+\in\mathbb{R}^d_{>0}$ such that $v_+^TA^+<0$.
\end{enumerate}
Moreover, when one of the above statements holds, then the reaction network $(\Xz,\mathcal{R})$ is robustly ergodic.
\end{proposition}
\begin{proof}
The proof that (a) is equivalent to (b) simply follows from Theorem \ref{th:uni:erg}. The proof that (b) implies (c) is immediate. The converse can be proved using the fact that for two Metzler matrices $M_1,M_2\in\mathbb{R}^{d\times d}$ verifying the inequality $M_1\le M_2$, we have that $\lambda_F(M_1)\le\lambda_F(M_2)$ where $\lambda_F(\cdot)$ denotes the Frobenius eigenvalue (see e.g. \cite{Berman:94}). Hence, we have that $\lambda_F(M)\le\lambda_F(A^+)<0$ for all $M\in\mathcal{A}$. The conclusion then readily follows.
\end{proof}}

\subsection{Interval ergodicity of bimolecular reaction networks}

We now provide an extension of the conditions of Theorem \ref{th:bi:erg} for bimolecular networks to the case of uncertain networks described by uncertain matrices:
\begin{proposition}\label{prop:interval:ergodicity}
  Let us consider an uncertain open irreducible bimolecular reaction network $(\X{},\mathcal{R})$ with interval characteristic matrix $[A]$ and assume that there exists a vector $v\in\mathbb{R}^d_{>0}$ such that
  \begin{equation}
    v^TA^+<0\textnormal{ and }v^TS_b=0.
  \end{equation}
  Then, the stochastic reaction network $(\X{},\mathcal{R})$ is robustly exponentially ergodic for all $A\in[A^-,A^+]$.
\end{proposition}
\begin{proof}
  The result immediately follows from Theorem \ref{th:bi:erg} and Proposition \ref{lem:frob}.
\end{proof}

\subsection{Interval output-controllability of unimolecular reaction networks}

\blue{Let us consider the set of matrices \eqref{eq:calA} where the bounds $A^-,A^+$ have been determined such that \eqref{eq:intervalmodel} holds. Then, we have the following result:
\begin{proposition}\label{prop:interval:controllability}
The following statements are equivalent:
\begin{enumerate}[(a)]
  \item The interval system $([A],e_i,e_j^T)$ is interval output-controllable;
  \item For all $A\in\mathcal{A}$, there exists a vector $w\in\mathbb{R}^d_{\ge0}$ and a scalar $\mu\ge0$ such that $w^Te_i>0$ and $w^T(A-\mu I_d)+e_j^T=0$;
  \item There exists a vector $w_-\in\mathbb{R}^d_{\ge0}$ and a scalar $\mu_-\ge0$ verifying $w_-^Te_i>0$ and $w_-^T(A^--\mu_- I_d)+e_j^T=0$. \mendth
\end{enumerate}
Moreover, if all the matrices in $\mathcal{A}$ are Hurwitz stable, then the above statements hold with $\mu=\mu_-=0$.
\end{proposition}}
\blue{\begin{proof}
  The equivalence between the two first statements follows from Proposition \ref{prop:outputcontrollability}. Clearly, (b) implies (c). We prove that the converse is also true. To this aim, define $A(\Delta)$ as $A(\Delta):=A^-+\Delta$ where $\Delta\in \Deltab:=[0,A^+-A^-]$. The key idea is to build a $w(\Delta)\ge0$ and a $\mu(\Delta)\ge0$ that verify the expressions $w(\Delta)^T(A(\Delta)-\mu(\Delta) I_d)+e_\ell^T=0$ and $w(\Delta)^Te_1>0$ for all $\Delta\in\Deltab$ provided that $w_-^T(A^--\mu_-I_d)+e_\ell^T=0$ and $w_-^Te_i>0$. We prove that such a $w(\Delta)$ is given by $w(\Delta):=[(A^--\mu_-I_d)(A(\Delta)-\mu(\Delta)I_d)]^{-1})^Tw_-$ with some large enough $\mu_-\ge0$ such that $\mu_-\ge\mu(\Delta)$ with $A(\Delta)-\mu(\Delta)I_d$ Hurwitz stable. Note that such a $\mu(\Delta)$ always exists.

  We first prove that this $w(\Delta)$ is nonnegative and that it verifies $e_1^Tw(\Delta)>0$ for all $\Delta\in\Deltab$. To show this, let us rewrite it as $w(\Delta)=[I_d+((\mu(\Delta)-\mu_-)I_d-\Delta)(A(\Delta)-\mu(\Delta)I_d)]^Tw_-$. Under the considered assumptions for $\mu(\Delta)$ and $\mu_-$, we get that $(A(\Delta)-\mu(\Delta)I_d)^{-1}\le0$. This together with $\Delta\ge0$  and $\mu_--\mu(\Delta)\ge0$, we obtain that $w(\Delta)\ge w_-\ge0$ for all $\Delta\in\Deltab$ and, therefore, that $w(\Delta)^Te_1\ge w_-^Te_1>0$ for all $\Delta\in\Deltab$.

  We now show that this $w(\Delta)$ verifies the equality condition. Substituting the expression for  $w(\Delta)$ in $w(\Delta)^T(A(\Delta)-\mu(\Delta)I_d)$ yields
%
%of the first statement holds, i.e. $w_-^TA^-+e_\ell^T=0$. Evaluating then $w(\Delta)^T(A^-+\Delta)$ yields that
\begin{equation}
%  \begin{array}{lcl}
    w(\Delta)^T(A(\Delta)-\mu(\Delta)I_d) %&=& w_-^T(A^--\mu_-I_d)(A^-+\Delta-\mu(\Delta)I_d)^{-1}(A^-+\Delta-\mu(\Delta)I_d)\\
    = w_-^T(A^--\mu_-I_d)=-e_\ell^T
     %                                               &=&-e_\ell^T
 % \end{array}
\end{equation}
where the last equality has been obtained from the assumption that $w_-^T(A^--\mu_-I_d)+e_\ell^T=0$. This proves that (c) implies (b). The final  statement follows from the same reasons as in Proposition \ref{prop:outputcontrollability}.
\end{proof}}

\subsection{Antithetic integral control of interval reaction networks}

We are now in position to state the following generalization of Theorem \ref{th:affine:nominal}:
\begin{theorem}\label{th:affine:interval}
Let us consider an open unimolecular reaction network $(\Xz,\mathcal{R})$ with interval characteristic matrix $[A]$ and interval offset vector $[b_0]=[b_0^-,b_0^+]$. Assume also that the closed-loop reaction network $((\Xz,\Zz),\mathcal{R}\cup\mathcal{R}^c)$ is irreducible. Then, the following statements are equivalent:
  \begin{enumerate}[(a)]
     \item The open-loop reaction network $(\Xz,\mathcal{R})$ is interval ergodic and the system  $([A],e_1,e_\ell^T)$ is interval output controllable.
    \item There exist two vectors $v_+\in\mathbb{R}_{>0}^d$, $w_-\in\mathbb{R}_{\ge0}^d$ such that $v_+^TA^+<0$, $w_-^Te_1>0$ and $w_-^TA^-+e_\ell^T=0$.
  \end{enumerate}
     Moreover, when one of the above statements holds, then the closed-loop reaction network $((\Xz,\Zz),\mathcal{R}\cup\mathcal{R}^c)$ is interval ergodic and we have that $\E[X_\ell(t)]\to\mu/\theta$ as $t\to\infty$  for any values for the controller rate parameters $\eta,k>0$ provided that
\begin{equation}\label{eq:lowerbound2}
      \dfrac{\mu}{\theta}>\dfrac{q^T(A^+-\Delta)^{-1}b^+}{\alpha q^T(A^+-\Delta)^{-1}e_\ell}
    \end{equation}
    and
    \begin{equation}
      q^T(\alpha(A^+-\Delta)^{-1}+I_d)\ge0
    \end{equation}
for some $\alpha>0$, $q\in\mathbb{R}^d_{>0}$ and for all $\Delta\in[0,A^+-A^-]$.\mendth
\end{theorem}
\blue{\begin{proof}
The proof of the equivalence between (a) and (b) follows from the notion of interval ergodicity and interval output controllability as well as Proposition \ref{prop:interval:ergodicity} and Proposition \ref{prop:interval:controllability}. The conclusion of the theorem is an adaptation of that of Theorem \ref{th:affine:nominal}. To prove \eqref{eq:lowerbound2}, let us define, with some slight abuse of notation, the matrix $A(\Delta):=A^+-\Delta$, $\Delta\in \Deltab:=[0,A^+-A^-]$. This (Metzler) matrix is Hurwitz stable for all $\Delta\in \Deltab$ and, therefore, its inverse is nonpositive. We need now to construct a suitable positive vector $v(\Delta)\in\mathbb{R}_{>0}^d$ such that $v(\Delta)^TA(\Delta)<0$ for all $\Delta\in\Deltab$ provided that $v_+^TA^+<0$. We prove now that such a $v(\Delta)$ is given by $v(\Delta)=(A^+(A^+-\Delta)^{-1})^Tv_+$. To prove its positivity for all $\Delta\in\Deltab$, first rewrite $v(\Delta)$ as $v(\Delta)=(A^+-\Delta)^{-T}(v_+^TA^+)^T$ and note that we both have $(A^+-\Delta)^{-1}\le0$ and $v_+^TA^+$, hence the resulting vector is strictly positive for all $\Delta\in\Deltab$. We show now that this vector verifies $v(\Delta)^TA(\Delta)<0$ for all $\Delta\in\Deltab$. Direct substitution yields $v(\Delta)^TA(\Delta)=v_+^TA^+<0$, which proves the desired result. To obtain a more explicit expression for $v(\Delta)$, note that since $A^+$ is Metzler and Hurwitz stable, then for any $q\in\mathbb{R}^d_{>0}$, there exists a  $v_+\in\mathbb{R}^d_{>0}$ such that $v_+^TA^+=-q^T$. Substituting the expression $v_+^T=-q^T(A^+)^{-1}$ in the above formula for $v(\Delta)$ and substituting into \eqref{eq:lowerbound} yields \eqref{eq:lowerbound2}.
\end{proof}}

As in the nominal case, the above result can be exactly formulated as the linear program
 \begin{equation}\label{feas:int}
    \begin{array}{rcl}
      \textnormal{Find} && v\in\mathbb{R}_{>0}^d,\ w\in\mathbb{R}_{\ge0}^d\\
      \textnormal{s.t.}&&w^Te_1>0,v^TA^+<0,w^TA^-+e_\ell^T=0
    \end{array}
  \end{equation}
  which has exactly the same complexity as the linear program \eqref{feas:nom}. Hence, checking the possibility of controlling a family of networks defined by a characteristic interval-matrix is not more complicated that checking the possibility of controlling a single network.

\section{Robust results}\label{sec:robust}

To palliate the potential lack of accuracy of the interval approach, the robust approach captures the exact parameter dependence is developed in this section.

\dgreen{\subsection{Preliminaries}

The following lemma will be useful in proving the main results of this section:
\begin{lemma}\label{lem:det}
Let us consider a matrix $M(\theta)\in\mathbb{R}^{d\times d}$ which is Metzler and bounded for all $\theta\in\Theta\subset\mathbb{R}^N_{\ge0}$ and where $\Theta$ is assumed to be compact and connected. Then, the following statements are equivalent:
\begin{enumerate}[(a)]
  \item The matrix $M(\theta)$ is Hurwitz stable for all $\theta\in\Theta$.
  \item The coefficients of the characteristic polynomial of $M(\theta)$ are positive \blue{for all $\theta\in\Theta$.}
  \item The conditions hold:
  \begin{enumerate}[(c1)]
    \item there exists a $\theta^*\in\Theta$ such that $M(\theta^*)$ is Hurwitz stable, and
    \item for all $\theta\in\Theta$ we have that $(-1)^d\det(M(\theta))>0$.
  \end{enumerate}
\end{enumerate}
\end{lemma}
\begin{proof}
The proof of the equivalence between (a) and (b) follows, for instance, from \cite{Mitkowski:00} and is omitted. It is also immediate to prove that (b) implies (c) since if  $M(\theta)$ is Hurwitz stable for all $\theta\in\Theta$ then (c1) holds and the constant term of the characteristic polynomial of $M(\theta)$ is positive on $\theta\in\Theta$. Using now the fact that this constant term is equal to $(-1)^d\det(M(\theta))$ yields the result.

To prove that (c) implies (a), we use the contraposition. Hence, let us assume that there exists at least a $\theta_u\in\Theta$ for which the matrix $M(\theta_u)$ is not Hurwitz stable. If such a $\theta_u$ can be arbitrarily chosen in $\Theta$, then this implies the negation of statement (c1) (i.e. for all $\theta^*\in\Theta$ the matrix $M(\theta^*)$ is not Hurwitz stable) and the first part of the implication is proved.

Let us consider now the case where there exists some $\theta_s\in\Theta$ such that $M(\theta_s)$ is Hurwitz stable. Let us then choose a $\theta_u$ and a $\theta_s$ such that $M(\theta_u)$ is not Hurwitz stable and $M(\theta_s)$ is. Since $\Theta$ is connected, then there exists a path $\mathscr{P}\subset\Theta$ from $\theta_s$ and $\theta_u$. From Perron-Frobenius theorem, the dominant eigenvalue, denoted by $\lambda_{PF}(\cdot)$, is real and hence, we have that $\lambda_{PF}(M(\theta_s))<0$ and $\lambda_{PF}(M(\theta_u))\ge0$. Hence, from the continuity of eigenvalues then there exists a $\theta_c\in\mathscr{P}$ such that $\lambda_{PF}(M(\theta_c))=0$, which then implies that $(-1)^d\det(M(\theta_c))=0$, or equivalently, that the negation of (c2) holds. This concludes the proof.
\end{proof}

Before stating the next main result of this section, let us assume that $S_u$ in Definition \ref{def:Ab} has the following form
\begin{equation}
  S_u=\begin{bmatrix}
  S_{dg} & S_{ct} & S_{cv}
  \end{bmatrix}
\end{equation}
where $S_{dg}\in\mathbb{R}^{d\times n_{dg}}$ is a matrix with nonpositive columns, $S_{ct}\in\mathbb{R}^{d\times n_{ct}}$ is a matrix with nonnegative columns and $S_{cv}\in\mathbb{R}^{d\times n_{cv}}$ is a matrix with columns containing at least one negative and one positive entry. Also, decompose accordingly $\rho_u$ as $\rho_u=:\col(\rho_{dg},\rho_{ct},\rho_{cv})$ and define

\begin{equation*}
  \rho_{\bullet}\in\mathcal{P}_{\bullet}:=[\rho_{\bullet}^-,\rho_{\bullet}^+],\ 0\le\rho_{\bullet}^-\le\rho_{\bullet}^+<\infty
\end{equation*}
where $\bullet\in\{dg,ct,cv\}$ and let $\mathcal{P}_u:=\mathcal{P}_{dg}\times\mathcal{P}_{ct}\times\mathcal{P}_{cv}$.

In this regard, we can alternatively rewrite the matrix $A(\rho_u)$ as $A(\rho_{dg},\rho_{ct},\rho_{cv})$. We then have the following result:
\begin{lemma}\label{lem:A+}
The following statements are equivalent:
  \begin{enumerate}[(a)]
       \item The matrix $A(\rho_u)$ is Hurwitz stable for all $\rho_u\in\mathcal{P}_u$.
    \item The matrix
    \begin{equation}
          A^+(\rho_{cv}):= A(\rho_{dg}^-,\rho_{ct}^+,\rho_{cv})
    \end{equation}
     is Hurwitz stable for all $\rho_{cv}\in\mathcal{P}_{cv}$.
  \end{enumerate}
\end{lemma}
\begin{proof}
The proof that (a) implies (b) is immediate. To prove that (b) implies (a), first note that we have
\begin{equation}
  A(\rho_{dg},\rho_{ct},\rho_{cv})\le A^+(\rho_{cv})=A(\rho_{dg}^-,\rho_{ct}^+,\rho_{cv})
\end{equation}
since for all $(\rho_{dg},\rho_{ct},\rho_{cv})\in\mathcal{P}_{u}$. Using the fact that for two Metzler matrices $B_1,B_2$, the inequality $B_1\le  B_2$ implies  $\lambda_{PF}(B_1)\le \lambda_{PF}(B_2)$ \cite{Berman:94}, then we can conclude that $A(\rho_{dg}^-,\rho_{ct}^+,\rho_{cv})$ is Hurwitz stable for all  $\rho_{cv}\in\mathcal{P}_{cv}$ if and only if  the matrix $A(\rho_{dg},\rho_{ct},\rho_{cv})$ is Hurwitz stable for all $(\rho_{dg},\rho_{ct},\rho_{cv})\in\mathcal{P}_{u}$. This completes the proof.
\end{proof}}

\subsection{Robust ergodicity of unimolecular networks}

The following theorem states the main result on the robust ergodicity of unimolecular reaction networks:
\begin{proposition}\label{th:uni_rob}\label{prop:robust:ergodicity}
Let us consider an irreducible open unimolecular reaction network $(\X{},\mathcal{R})$ with parameter-dependent characteristic matrix $A(\rho_u)$, $\rho_u\in\mathcal{P}_u$. Then, the following statements are equivalent:
 \begin{enumerate}[(a)]
   \item The reaction network $(\X{},\mathcal{R})$ is robustly ergodic.
   \item The matrix $A(\rho_u)$ is Hurwitz stable for all $\rho_u\in\mathcal{P}_u$.
    \item The matrix
    \begin{equation}
          A^+(\rho_{cv}):= A(\rho_{dg}^-,\rho_{ct}^+,\rho_{cv})
    \end{equation}
     is Hurwitz stable for all $\rho_{cv}\in\mathcal{P}_{cv}$.
     \item There exists a $\rho_{cv}^s\in\mathcal{P}_{cv}$ such that the matrix $A^+(\rho_{cv}^s)$ is Hurwitz stable and the polynomial $(-1)^d\det(A^+(\rho_{cv}))$ is positive for all $\rho_{cv}\in\mathcal{P}_{cv}$.
 %  %
    \item There exists a polynomial vector-valued function ${v:\mathcal{P}_{cv}\mapsto\mathbb{R}_{>0}^d}$ of degree at most $d-1$ such that $v(\rho_{cv})^TA^+(\rho_{cv})<0$ for all $\rho_{cv}\in\mathcal{P}_{cv}$.
 \end{enumerate}
\end{proposition}
\begin{proof}
The equivalence between the statement (a), (b) and (c) directly follows from Lemma \ref{lem:det} and Lemma \ref{lem:A+}. To prove the equivalence between the statements (b) and (d), first remark that (b) is equivalent to the fact that for any $q(\rho_{cv})>0$ on $\mathcal{P}_{cv}$, there exists a unique parameterized vector $v(\rho_{cv})\in\mathbb{R}^d$ such that
$v(\rho_{cv})>0$ and $v(\rho_{cv})^TA^+(\rho_{cv})=-q(\rho_{cv})^T$ for all $\rho_{cv}\in\mathcal{P}_{cv}$. Choosing $q(\rho_{cv})=-\mathds{1}_n(-1)^d\det(A^+(\rho_{cv}))$, we get that such a $v(\rho_{cv})$ is given by
\begin{equation}
\begin{array}{rcl}
    v(\rho_{cv})^T      &=&        -\mathds{1}_d^T(-1)^d\det(A^+(\rho_{cv}))A^+(\rho_{cv})^{-1}\\
                                &=&       (-1)^{d+1}\mathds{1}^T_d\textnormal{Adj}(A^+(\rho_{cv}))>0
\end{array}
\end{equation}
for all $\rho_{cv}\in\mathcal{P}_{cv}$. Since the matrix $A^+(\rho_{cv})$ is affine in $\rho_{cv}$, then the adjugate matrix $\textnormal{Adj}(A^+(\rho_{cv})$ contains entries of at most degree $d-1$ and the conclusion follows.
\end{proof}

\subsection{Robust ergodicity of bimolecular networks}

In the case of bimolecular networks, we have the following result:
\begin{proposition}
Let us consider an irreducible open bimolecular reaction network $(\X{},\mathcal{R})$ with parameter-dependent characteristic matrix $A(\rho_u)$, $\rho_u\in\mathcal{P}_u$. Then, the following statements are equivalent:
 \begin{enumerate}[(a)]
 %  \item The matrix $A(\rho)$ is Hurwitz stable for all $\rho\in\mathcal{P}$.
   %
\item There exists a polynomial vector-valued function ${v:\mathcal{P}_{u}\mapsto\mathbb{R}_{>0}^d}$ of degree at most $d-1$ such that
   \begin{equation}\label{eq:dmlsqdksdsmdkml}
     v(\rho_u)>0,\ v(\rho_u)^TS_b=0\ \textnormal{and}\ v(\rho_u)^TA(\rho_u)<0
   \end{equation}
for all $\rho_u\in\mathcal{P}_u$.
\item There exists a polynomial vector-valued function ${\tilde{v}:\mathcal{P}_{cv}\mapsto\mathbb{R}^{d-n_b}}$ of degree at most $d-1$ such that
     \begin{equation}
     \tilde{v}(\rho_{cv})^TS_b^\bot>0\ \textnormal{and}\ \tilde{v}(\rho_{cv})^TS_b^\bot A^+(\rho_{cv})<0
   \end{equation}
for all $\rho_{cv}\in\mathcal{P}_{cv}$ and where $n_b:=\rank(S_b)$ and $S_b^\bot S_b=0$, $S_b^\bot$ full-rank.
 \end{enumerate}
 Moreover, when one of the above statements holds, then the network  $(\X{},\mathcal{R})$ is robustly ergodic.
\end{proposition}
\begin{proof}
  It is immediate to see that (a) implies (b). To prove the converse, first note that we have that $v(\rho_{cv})=(S_b^\bot)^T\tilde{v}(\rho_{cv})$ verifies $v(\rho_{cv})^TS_b=0$ and $v(\rho_{cv})>0$ for all $\rho_{cv}\in\mathcal{P}_{cv}$. This proves the equality and the first inequality in \eqref{eq:dmlsqdksdsmdkml}. Observe now that for any $\rho_u\in\mathcal{P}_u$, there exists a nonnegative matrix $\Delta(\rho_{dg},\rho_{ct})\in\mathbb{R}^{d\times d}_{\ge0}$ such that $A(\rho_u)=A^+(\rho_{cv})-\Delta(\rho_{dg},\rho_{ct})$. Hence, we have that
  \begin{equation}
  \begin{array}{rcl}
        v(\rho_{cv})^TA(\rho_u)&=&v(\rho_{cv})^T(A^+(\rho_{cv})-\Delta(\rho_{dg},\rho_{ct}))\\
        &\le&v(\rho_{cv})^TA^+(\rho_{cv})<0
  \end{array}
  \end{equation}
  which proves the implication.
\end{proof}

As in the unimolecular case, we have been able to reduce the number of parameters by using an upper-bound on the characteristic matrix. It is also interesting to note that the condition $\tilde{v}(\rho_{cv})^TS_b^\bot A^+(\rho_{cv})<0$ can be sometimes brought back to a problem of the form $\tilde{v}(\rho_{cv})^TM(\rho_{cv})<0$ for some square, and sometimes Metzler, matrix $M(\rho_{cv})$ which can be dealt in the same way as in the unimolecular case.

\subsection{Robust output controllability of unimolecular networks}

\blue{In this case, we have the following result:
\begin{proposition}\label{prop:robust:controllability}
The following statements are equivalent:
\begin{enumerate}[(a)]
  \item The parameter-dependent system $(A(\rho_u),e_i,e_j^T)$, $\rho_u\in\mathcal{P}_u$, is robustly output-controllable;
  \item For all $\rho_u\in\mathcal{P}_u$, there exists a vector-valued function $w:\mathcal{P}_u\mapsto\mathbb{R}^d_{\ge0}$ and a function $\mu:\mathcal{P}_u\mapsto\mathbb{R}_{\ge0}$
      verifying $w(\rho_u)^Te_i>0$ and $w(\rho_u)^T(A(\rho_u)-\mu(\rho_u)I_d)+e_j^T=0$ for all $\rho_u\in\mathcal{P}_u$; and
%  \item There exists a vector-valued function $w_-:\mathcal{P}_{cv}\mapsto\mathbb{R}^d_{\ge0}$ and a function $\mu_-:\mathcal{P}_{cv}\mapsto\mathbb{R}_{\ge0}$  verifying $w_-(\rho_{cv})^Te_i>0$ and $w_-(\rho_{cv})^T(A^-(\rho_{cv})-\mu_-(\rho_{cv})I_d)+e_j^T=0$ where $A^-(\rho_{cv}):=A(\rho_{dg}^+, \rho_{ct}^-,\rho_{cv})$
        \item There exists a vector $w_-\in\mathbb{R}^d_{\ge0}$ and a scalar $\mu_-\in\mathbb{R}_{\ge0}$  verifying $w_-^Te_i>0$ and $w_-^T(A^--\mu_-I_d)+e_j^T=0$ where $A^-:=A(\rho_{dg}^+, \rho_{ct}^-,\rho_{cv}^-)$.\mendth
\end{enumerate}
Moreover, when the matrix $A(\rho_u)$ is Hurwitz stable for all $\rho_u\in\mathcal{P}_u$, then the statement (b) holds with $\mu\equiv 0$.
\end{proposition}}
\blue{\begin{proof}
The proof of this result follows from the same lines as the proof of Proposition \ref{prop:interval:controllability} with the difference that the parameter $\rho_{cv}$ is now present. However, we know from Proposition \ref{prop:outputcontrollability} that only the location of the nonzero off-diagonal elements matters from the output-controllability. In this respect, the worst-case happens whenever the reaction rates of the catalytic and conversion reactions are the smallest. The final  statement follows from the same reasons as in Proposition \ref{prop:outputcontrollability}.
\end{proof}}

\subsection{Robust antithetic integral control of unimolecular networks}

\begin{theorem}\label{th:affine:robust}
Let us consider an open unimolecular reaction network $(\Xz,\mathcal{R})$ with parameter-dependent characteristic matrix $A(\rho_u)$, $\rho_u\in\mathcal{P}_u$ and parameter-dependent offset vector $b_0(\rho_0)$, $\rho_0\in\mathcal{P}_0$. Assume further that the closed-loop reaction network $((\Xz,\Zz),\mathcal{R}\cup\mathcal{R}^c)$ is irreducible. Then, the following statements are equivalent:
  \begin{enumerate}[(a)]
   \item The open-loop reaction network $(\Xz,\mathcal{R})$ is robustly ergodic and the system  $(A(\rho_u),e_1,e_\ell^T)$ is robustly output controllable.\label{st:rob1}
    \item There exist two vector-valued functions $v_+:\mathcal{P}_{cv}\mapsto\mathbb{R}_{>0}^d$, $w_-:\mathcal{P}_{cv}\mapsto\mathbb{R}_{>0}^d$ such that $v_+(\rho_{cv})^TA^+(\rho_{cv})<0$, $w_-(\rho_{cv})^Te_1>0$ and $w_-(\rho_{cv})^TA^-(\rho_{cv})+e_\ell^T=0$.\label{st:rob2}
  \end{enumerate}
     Moreover, when one of the above statements holds, then the closed-loop reaction network $((\Xz,\Zz),\mathcal{R}\cup\mathcal{R}^c)$ is robustly ergodic and we have that $\E[X_\ell(t)]\to\mu/\theta$ as $t\to\infty$  for any values for the controller rate parameters $\eta,k>0$ provided that
\begin{equation}\label{eq:lowerbound3}
      \dfrac{\mu}{\theta}>\dfrac{v^+(\rho_{cv})^TA^+(\rho_{cv})}{\alpha v^+(\rho_{cv})^Te_\ell}
    \end{equation}
    and
    \begin{equation}
      v^+(\rho_{cv})^T(A^+(\rho_{cv})+\alpha I_d)\le0
    \end{equation}
for some $\alpha>0$ and for all $\rho_{cv}\in\mathcal{P}_{cv}$.
\end{theorem}
\begin{proof}
  The equivalence between the statements (a) and (b) follows from Proposition \ref{prop:robust:ergodicity} and Proposition \ref{prop:robust:controllability}. The concluding statement is an adaptation of that of Theorem \ref{th:affine:nominal}.
\end{proof}
\blue{The above result can be exactly formulated as the following  infinite-dimensional linear program
 \begin{equation}
    \begin{array}{rcl}
      \textnormal{Find} && v:\mathcal{P}_{cv}\mapsto\mathbb{R}_{>0}^d,\ w:\mathcal{P}_{cv}\mapsto\mathbb{R}_{\ge0}^d\\
      \textnormal{s.t.}&&w(\rho_{cv})^Te_1>0,v(\rho_{cv})^TA^+(\rho_{cv})<0,\\
      &&w(\rho_{cv})^TA^-(\rho_{cv})+e_\ell^T=0, \textnormal{ for all }\rho_{cv}\in\mathcal{P}_{cv}
    \end{array}
  \end{equation}
 which has a higher complexity than the previous feasibility problems due to its infinite-dimensional nature. However, from Proposition \ref{prop:robust:ergodicity} and Proposition \ref{prop:robust:controllability}, we know that it is enough to look for polynomials of degree $d-1$. Hence, polynomial optimization methods can be used to solve this problem; see e.g. \cite{Handelman:88,Briat:11h,Putinar:93,Parrilo:00,Lasserre:01,Lasserre:10}.}

\section{Sign results}\label{sec:sign}

The objective of this section is to prove analogues of the ergodicity and output-controllability results of Section \ref{sec:prel} whenever the parameters take arbitrary positive values.

\subsection{Sign-ergodicity of unimolecular networks}

The following result proved in \cite{Briat:14c} will turn out to be a key ingredient for deriving the main result of this section:
\begin{proposition}[\cite{Briat:14c}]\label{prop:sign:ergodicity}
Let us consider an irreducible open unimolecular reaction network $(\X{},\mathcal{R})$ with sign characteristic matrix $S_A$. Then, the following statements are equivalent:
  \begin{enumerate}[(a)]
   \item The reaction network $(\X{},\mathcal{R})$ is sign-ergodic.
    \item All the matrices in $\mathcal{Q}(S_A)$ are Hurwitz stable.
    \item The matrix $\sgn(S_A)$ is Hurwitz stable.
    \item The diagonal elements of $\Sigma$ are negative and the directed graph $\mathcal{G}_{S_A}=(\mathcal{V},\mathcal{E})$ defined with
    \begin{itemize}
      \item $\mathcal{V}:=\{1,\ldots,d\}$ and
      \item $\mathcal{E}:=\{(m,n):\ e_n^T\Sigma e_m\ne0,\ m,n\in V,\ m\ne n\}$.
    \end{itemize}
    is an acyclic directed graph.\mendth
  \end{enumerate}
\end{proposition}

\subsection{Sign output-controllability of unimolecular networks}

\blue{We have the following result:
\begin{proposition}\label{prop:sign:controllability}
The following statements are equivalent:
\begin{enumerate}[(a)]
  \item The sign system $(S_A,e_i,e_j^T)$ is sign output-controllable;
  \item For all $A\in\mathcal{Q}(S_A)$, there exists a vector $w\in\mathbb{R}^d_{\ge0}$ and a scalar $\mu\in\mathbb{R}_{\ge0}$ verifying $w^Te_i>0$ and $w^T(A-\mu I_d)+e_j^T=0$; and
  \item There exists a vector $w\in\mathbb{R}^d_{\ge0}$ and a scalar $\mu\in\mathbb{R}_{\ge0}$  verifying $w^Te_i>0$ and $w^T(\sgn(A)-\mu I_d)+e_j^T=0$.\mendth
\end{enumerate}
Moreover, if all the matrices in $\mathcal{Q}(S_A)$ are Hurwitz stable then the above statements hold with $\mu=0$.
\end{proposition}}
\blue{\begin{proof}
  The equivalence between the statements (a) and (b) follows from Proposition \ref{prop:outputcontrollability} and the definition of sign output controllability. The implication that (b) implies (c) is also immediate. To show the reverse direction, it is enough to notice that the output-controllability only depends on the location of the nonzero off-diagonal entry (see Proposition \ref{prop:outputcontrollability}, (e)) and is therefore a structural property. The final statement follows from Proposition \ref{prop:outputcontrollability}.
\end{proof}}

\subsection{Sign antithetic integral control of unimolecular networks}

We are now ready to state the main result of this section:
\begin{theorem}\label{th:affine:structural}
Let us consider an open unimolecular reaction network $(\Xz,\mathcal{R})$ with sign characteristic matrix $S_A$ and sign offset vector $S_b$. Assume further that the closed-loop reaction network $((\Xz,\Zz),\mathcal{R}\cup\mathcal{R}^c)$ is irreducible. Then, the following statements are equivalent:
  \begin{enumerate}[(a)]
      \item The open-loop reaction network $(\Xz,\mathcal{R})$ is sign ergodic and the system $(S_A,e_1,e_\ell^T)$ is sign output controllable.\label{st:sign1}
  \item There exist vectors $v\in\mathbb{R}_{>0}^d$, and $w\in\mathbb{R}_{\ge0}^{d}$, $w_1>0$  such that  the conditions\label{st:st:1}
    \begin{equation}
        v^T\sgn(S_A)<0\textnormal{ and }w^T\sgn(S_A)+e_\ell^T=0.
    \end{equation}
    hold.
  \end{enumerate}
    Moreover, when one of the above statements holds, then the closed-loop reaction network $((\Xz,\Zz),\mathcal{R}\cup\mathcal{R}^c)$ is sign ergodic and we have that $\E[X_\ell(t)]\to\mu/\theta$ as $t\to\infty$  for any values for the controller rate parameters $\eta,k>0$ and each $(A,b_0)\in\mathcal{Q}(S_A)\times \mathcal{Q}(S_b)$ provided that
    \begin{equation}\label{eq:lowerbound3}
      \dfrac{\mu}{\theta}>\dfrac{v^Tb_0}{\alpha e_\ell^Tv}
    \end{equation}
   where $\alpha>0$ and $v\in\mathbb{R}_{>0}^d$ verify $v^T(A+cI)\le0$.\mendth
\end{theorem}
\begin{proof}
The equivalence between the statements (a) and (b) follows from Proposition \ref{prop:sign:ergodicity} and Proposition \ref{prop:sign:controllability}. The conclusion follows from an adaptation of that of Theorem \ref{th:affine:nominal}.
\end{proof}
\blue{The above result naturally translates into the following linear program
\begin{equation}\label{feas:sign}
    \begin{array}{rcl}
      \textnormal{Find} && v\in\mathbb{R}_{>0}^d,\ w\in\mathbb{R}_{\ge0}^{d}\\
      \textnormal{s.t.}&&w^Te_1>0,v^T\sgn(S_A)<0, w^T\sgn(S_A)+e_\ell^T=0
    \end{array}
  \end{equation}
   which has the same complexity as in the nominal and the interval case.}

\section{Structural results}\label{sec:structural}

Similarly as for the interval approach, the sign approach fails to be tight in the presence of conversion reactions. The structural approach is developed here to complement this.

\subsection{A preliminary result}

The following result will play an instrumental role in proving the results in this section:
\begin{lemma}\label{lem:rescaling}
 Let $A(\rho_u)\in\mathbb{R}^{d\times d}$ be the characteristic matrix of some unimolecular network and $\rho_u\in\mathcal{P}_u$. Then, the following statements are equivalent:
  \begin{enumerate}[(a)]
    \item For all $\rho_{dg}\in\mathcal{P}_{dg}$ and a $\rho_{cv}\in\mathcal{P}_{cv}$, the matrix $A(\rho_{dg},\rho_{cv},0)$ is Hurwitz stable.
    \item The matrix $A(\mathds{1},\rho_{cv},0)$ is Hurwitz stable for all $\rho_{cv}\in\mathcal{P}_{cv}$.
  \end{enumerate}
\end{lemma}
\begin{proof}
  The proof that (a) implies (b) is immediate. To prove the reverse implication, we use contraposition and we assume that there exist a $\rho_{dg}\in\mathcal{P}_{dg}$ and a $\rho_{cv}\in\mathcal{P}_{cv}$ such that $A(\rho_{dg},\rho_{cv},0)$ is not Hurwitz stable. Then, we clearly have that
  \begin{equation}
    A(\rho_{dg},\rho_{cv},0)\le A(\theta \mathds{1},\rho_{cv},0)
  \end{equation}
  where $\theta=\min(\rho_{dg})$ and hence $A(\theta \mathds{1},\rho_{cv},0)$ is not Hurwitz stable. Since $A(\theta \mathds{1},\rho_{cv},0)$ is affine in $\theta$ and $\rho_{cv}$, then we have that  $\theta A(\mathds{1},\rho_{cv}/\theta,0)$ and since $\theta$ is independent of $\rho_{cv}$, then we get that the matrix $A(\mathds{1},\tilde{\rho}_{cv},0)$ is not Hurwitz stable for some $\tilde{\rho}_{cv}\in\mathcal{P}_{cv}$ . The proof is complete.
\end{proof}

\subsection{Structural ergodicity of unimolecular networks}

We have the following result:
\begin{proposition}\label{prop:structural:ergodicity}
Let us consider an open  irreducible unimolecular reaction network $(\X{},\mathcal{R})$ with parameter-dependent characteristic matrix $A(\rho_u)$, $\rho_u\in\mathcal{P}_u$. Then, the following statements are equivalent:
\begin{enumerate}[(a)]
\item The reaction network $(\X{},\mathcal{R})$ is structurally ergodic.
\item The matrix $A(\rho_u)$ is Hurwitz stable for all $\rho_u\in\mathbb{R}^{n_u}_{>0}$.
\item There exists a polynomial vector $v(\rho_u)\in\mathbb{R}^d$ of degree at most $d-1$ such that $v(\rho_u)>0$ and $v^TA(\rho_u)<0$ for all $\rho_u\in\mathbb{R}^{n_u}_{>0}$.
%  %
  \item There exists a $\rho_u^s\in\mathbb{R}^{n_u}_{>0}$ such that the matrix $A(\rho_u^s)$ is Hurwitz stable and the polynomial $(-1)^d\det(A(\rho_u))$ is positive for all $\rho_u\in\mathbb{R}^{n_u}_{>0}$.
\item For all $\rho_{dg}\in\mathbb{R}^{n_{dg}}_{>0}$ and a $\rho_{cv}\in\mathbb{R}^{n_{cv}}_{>0}$, the matrix $A_\rho:=A(\rho_{dg},\rho_{cv},0)$ is Hurwitz stable and we have that $\varrho(W_{ct}A_\rho^{-1}S_{ct})=0$.
    \item The matrix $A_n(\rho_{cv}):=A(\mathds{1},\rho_{cv},0)$ is Hurwitz stable for all $\rho_{cv}\in\mathbb{R}^{n_{cv}}_{>0}$ and $\varrho(W_{ct}A_n(\rho_{cv})^{-1}S_{ct})=0$ for all $\rho_{cv}\in\mathbb{R}^{n_{cv}}_{>0}$.
\end{enumerate}
Moreover, when each column of $S_{cv}$ contains exactly one entry equal to $-1$ and  one equal to 1, then the above statements are also equivalent to
\begin{enumerate}[(a)]
\setcounter{enumi}{5}
    \item The matrix $A_{\mathds{1}}:=A(\mathds{1},\mathds{1},0)$ is Hurwitz stable and $\varrho(W_{ct}A_{\mathds{1}}^{-1}S_{ct})=0$.
   \blue{ \item There exist vectors  $v_c\in\mathbb{R}_{>0}^d$, $v_d\in\mathbb{R}_{>0}^d$, $w\in\mathbb{R}_{\ge0}^d$  such that $v_c^TA_{\mathds{1}}<0$ and $v_d^T(\sgn(W_{ct}A_{\mathds{1}}^{-1}S_{ct})-I_d)<0$.}
        \end{enumerate}
\end{proposition}
\begin{proof}
The equivalence between the three first statements has been proved in Proposition \ref{prop:robust:ergodicity}. Let us prove now that (c) implies (d). Assuming that (c) holds, we get that the existence of a $\rho_u^s=\col(\rho_{dg}^s,\rho_{cv}^s,\rho_{ct}^s)$ such that the matrix $A(\rho_u^s)$ is Hurwitz stable immediately implies that the matrix $A_\rho=A(\rho_{dg},\rho_{cv},0)$ is Hurwitz stable since we have that $A_\rho\le A(\rho_u)$ and, therefore $\lambda_{PF}(A_\rho)\le \lambda_{PF}(A(\rho_u))<0$.
Using now the determinant formula, we have that
\begin{equation}
    \det(A(\rho_u))=\det(A_\rho)\det(I+D(\rho_{ct})W_{ct}A_\rho^{-1}S_{ct})
\end{equation}
where $D(\rho_{ct}):=\diag(\rho_{ct})$ and $W_{ct}$ is defined such that ${\diag(\rho_{ct})W_{ct}x}$ is the vector of propensity functions associated with the catalytic reactions. Hence, this implies that
\begin{equation}
\det(I+D(\rho_{ct})W_{ct}A_\rho^{-1}S_{ct})>0
\end{equation}
for all $\rho_{ct}\in\mathbb{R}^{n_{ct}}_{>0}$. Since the matrices $W_{ct},S_{ct}$ are nonnegative, the diagonal entries of $D(\rho_{ct})$ are positive and $A_\rho^{-1}$ is nonpositive (since $A_\rho$ is Metzler and Hurwitz stable), then it is necessary that all the eigenvalues of $W_{ct}A_\rho^{-1}S_{ct}$ be zero for the determinant to remain positive. This completes the argument.

The converse (i.e. (d) implies (c)) can be proven by noticing that if $A_\rho$ is Hurwitz stable, then $A_\rho+\epsilon S_{ct}W_{ct}$ remains Hurwitz stable for some sufficiently small $\epsilon>0$. This proves the existence of a  $\rho_u^s\in\mathbb{R}^d_{>0}$ such that the matrix $A(\rho_u^s)$.  Using the determinant formula, it is immediate to see that the second statement implies the determinant condition of statement (c).

The equivalence between the statements (d) and (e) comes from Lemma \ref{lem:rescaling} and the fact that the sign-pattern of the inverse of a Hurwitz stable Metzler matrix is uniquely defined by its sign-pattern; see \cite{Briat:17LAA}.

Let us now focus on the equivalence between the statements (d) and (f) under the assumption that $S_{cv}$ contains exactly one entry equal to $-1$ and  one equal to 1. Assume w.l.o.g that $S_{dg}=\col(-I_{n_{dg}},0)$. Then, we have that $\mathds{1}_{d}^TA(\rho_{dg},\rho_{cv},0)=\begin{bmatrix}
  -\rho_{dg}^T & 0
\end{bmatrix}$. Hence, the function $V(z)=\mathds{1}_d^Tz$ is a weak Lyapunov function for the linear positive system $\dot{z}=A(\rho_{dg},\rho_{cv},0)z$. Invoking LaSalle's invariance principle, we get that the matrix is Hurwitz stable if and only if the matrix
\begin{equation}
  A^{22}(\rho_{dg},\rho_{cv}):=\begin{bmatrix}
    0\\ I
  \end{bmatrix}^TA(\rho_{dg},\rho_{cv},0)  \begin{bmatrix}
    0 \\ I
  \end{bmatrix}
\end{equation}
is Hurwitz stable for all $(\rho_{dg},\rho_{cv})\in\mathbb{R}^{n_{dg}}_{>0}\times\mathbb{R}^{n_{cv}}_{>0}$. Note that this is a necessary condition for the matrix $A(\rho_{dg},\rho_{cv},0) $ to be Hurwitz stable for all rate parameters values. Hence, this means that the stability of the matrix $A_\rho$ is equivalent to the Hurwitz stability of $A_{\mathds{1}}:=A(\mathds{1},\mathds{1},0)$. Finally, since $A^{22}(\rho_{dg},\rho_{cv})$ is Hurwitz stable, then we have that $\mathds{1}^TA^{22}(\mathds{1},\mathds{1})<0$.

\blue{Finally, the equivalence between (f) and (g) follows from Proposition \ref{prop:stability} and the fact that the nonnegative matrix $W_{ct}A_{\mathds{1}}^{-1}S_{ct}$ has a zero spectral radius if and only if all its diagonal elements are zero and its directed graph is acyclic \cite{Briat:17LAA}. This is equivalent to say that the matrix $\sgn(W_{ct}A_{\mathds{1}}^{-1}S_{ct})$ satisfies the same conditions or, equivalently that $\sgn(W_{ct}A_{\mathds{1}}^{-1}S_{ct})-I_d$ is Hurwitz stable. The conclusion then follows.}
\end{proof}

\subsection{Structural output controllability of unimolecular networks}

\blue{We have the following result:
\begin{proposition}\label{prop:structural:controllability}
The following statements are equivalent:
\begin{enumerate}[(a)]
  \item The parameter-dependent system $(A(\rho_u),e_i,e_j^T)$, $\rho_u\in\mathcal{P}_u^\infty$, is structurally output-controllable.
  \item For all $\rho_u\in\mathcal{P}_u^\infty$, there exist a vector-valued function $w:\mathcal{P}_u^\infty\mapsto\mathbb{R}^d_{\ge0}$ and a function $\mu:\mathcal{P}_u^\infty\mapsto\mathbb{R}_{\ge0}$ verifying $w(\rho_u)^Te_i>0$ and $w(\rho_u)^T(A-\mu(\rho_u) I_d)+e_j^T=0$; and
  \item There exists a vector $w\in\mathbb{R}^d_{\ge0}$ and a scalar $\mu\in\mathbb{R}_{\ge0}$  verifying $w^Te_i>0$ and $w^T(\sgn(A)-\mu I_d)+e_j^T=0$.\mendth
\end{enumerate}
Moreover, if all the matrices in $\{A(\rho_u):\rho_u\in\mathcal{P}_u^\infty\}$ are Hurwitz stable then the above statements hold with $\mu\equiv0$.
\end{proposition}}
\blue{\begin{proof}
  The equivalence between the statements (a) and (b) follows from Proposition \ref{prop:outputcontrollability} and the definition of sign output controllability. To show the equivalence with the statement (c), it is enough to notice that the output controllability only depends on the location of the nonzero off-diagonal entry (see Proposition \ref{prop:outputcontrollability}, (e)) and is therefore a structural property. The final statement follows from Proposition \ref{prop:outputcontrollability}.
\end{proof}}

\subsection{Structural antithetic integral control of unimolecular networks}

We are now ready to state the main result of this section:
\begin{theorem}\label{th:affine:structural}
Let us consider an open unimolecular reaction network $(\Xz,\mathcal{R})$ with parameter-dependent characteristic matrix $A(\rho_u)$, $\rho_u\in\mathcal{P}_u^\infty$ and parameter-dependent offset vector $b_0(\rho_0)$, $\rho_0\in\mathbb{R}_{>0}^{n_0}$. Assume further that each column of $S_{cv}$ contains exactly one entry equal to $-1$ and  one equal to 1 and that the closed-loop reaction network $((\Xz,\Zz),\mathcal{R}\cup\mathcal{R}^c)$ is irreducible. Then, the following statements are equivalent:
  \begin{enumerate}[(a)]
      \item The open-loop reaction network $(\Xz,\mathcal{R})$ is structurally ergodic and the system $(A(\rho_u),e_1,e_\ell^T)$ is structurally output controllable.\label{st:str1}
  \item There exist vectors  $v_c\in\mathbb{R}_{>0}^d$, $v_d\in\mathbb{R}_{>0}^d$, $w\in\mathbb{R}_{\ge0}^d$  such that  the conditions\label{st:st:2}
    \begin{equation}
    \begin{array}{c}
      v_c^TA_{\mathds{1}}<0,v_d^T(\sgn(W_{ct}A_{\mathds{1}}^{-1}S_{ct})-I_d)<0,\\
      w^Te_1>0, w^T\sgn(A)+e_\ell^T=0
    \end{array}
    \end{equation}
    hold.
\end{enumerate}
   Moreover, when one of the above statements holds, then the closed-loop reaction network $((\Xz,\Zz),\mathcal{R}\cup\mathcal{R}^c)$ is structurally ergodic and we have that $\E[X_\ell(t)]\to\mu/\theta$ as $t\to\infty$  for any values for the controller rate parameters $\eta,k>0$ and each $(\rho_u,\rho_0)\in\mathcal{P}_u^\infty\times \mathbb{R}_{>0}^{n_0}$ provided that
    \begin{equation}
      \dfrac{\mu}{\theta}>\dfrac{v^Tb_0(\rho_0)}{\alpha e_\ell^Tv}
    \end{equation}
   where $\alpha>0$ and $v\in\mathbb{R}_{>0}^d$ verify $v^T(A(\rho_u)+\alpha I_d)\le0$.\mendth
\end{theorem}
\begin{proof}
  The equivalence between the two statements follows from Proposition \ref{prop:structural:ergodicity} and Proposition \ref{prop:structural:controllability}. The concluding statement is an adaptation of that of Theorem \ref{th:affine:nominal}.
\end{proof}

\section{Examples}\label{sec:ex}

\subsection{SIR model: Structural ergodicity of a bimolecular network}

Let us consider the open irreducible stochastic SIR model considered in \cite{Briat:13i} described by the matrices
\begin{equation}
  A = \begin{bmatrix}
    -\rho_{dg}^1 & 0 & \rho_{cv}^2\\
    0 & -(\rho_{dg}^2+\rho_{cv}^1) & 0\\
    0 & \rho_{cv}^1 & -(\rho_{dg}^3+\rho_{cv}^2)
  \end{bmatrix},\ S_b=\begin{bmatrix}
-1\\
1\\
0
  \end{bmatrix}
\end{equation}
where all the parameters are positive. The constraint $v^TS_b=0$ enforces that $v=\tilde{v}^TS_b^\bot$, $\tilde{v}>0$, where $S_b^\bot=\begin{bmatrix}
1 & 1 & 0\\
0 & 0 & 1
\end{bmatrix}$. This leads to
\begin{equation}
  \tilde{v}^TS_b^\bot A<0\Leftrightarrow \tilde{v}^T\begin{bmatrix}
-(\rho_{dg}^2+\rho_{cv}^1) & \rho_{cv}^2\\
 \rho_{cv}^1 & -(\rho_{dg}^3+\rho_{cv}^2)
  \end{bmatrix}<0.
\end{equation}
Since the entries are not independent, the corresponding sign-matrix is not sign-stable whereas this matrix is clearly structurally stable. Similarly, the associated interval matrix mays fail to be interval stable even in the case when it would be robustly stable. The latter only holds if the diagonal entries are negative and bounded away from 0.

\subsection{Antithetic integral control of an uncertain unimolecular reaction network}

\blue{Let us consider an open unimolecular irreducible reaction network $(\X{},\mathcal{R})$ with characteristic matrix $A(\rho_u)$ given by
\begin{equation}
  \begin{bmatrix}
    -\rho_{dg}^1 & 0 & 0 & \rho_{ct}^3\\
    \rho_{ct}^1 & -\rho_{dg}^2-\rho_{cv}^1 & 0 & 0\\
    \rho_{ct}^2 & 0 & -\rho_{dg}^3-\rho_{cv}^2+\rho_{ct}^4 & \rho_{cv}^3\\
    0 & \rho_{cv}^1 & \rho_{cv}^2 & -\rho_{dg}^4-\rho_{cv}^3
  \end{bmatrix}
\end{equation}
with $\rho_u=(\rho_{dg},\rho_{cv},\rho_{ct})$. The goal is to act on the first species to control the last one. Hence, we have $b=e_1$ and $c=e_4$. We also assume that the set $\mathcal{P}_u$ is compact for simplicity. The following statements hold:
\begin{enumerate}[(Obs1)]
  \item The network is interval output-controllable if and only if $\rho_{ct}^{1-}\rho_{cv}^{1-}>0$ or $\rho_{ct}^{2-}\rho_{cv}^{2-}>0$.
  \item The network is robustly output-controllable under the same conditions.
  \item The network is sign output-controllable.
  \item The network is structurally output-controllable.
\end{enumerate}
We now focus on the ergodicity property of the associated network. The ergodicity conditions given below also preserve the output-controllability of the network:
\begin{enumerate}[(Erg1)]
  \item The network is interval ergodic if and only if the associated $A^+$ matrix is Hurwitz stable.
  \item The network is robustly ergodic if and only if $\rho_{dg}^{1-}\rho_{dg}^{4-}-\rho_{ct}^{3+}(\rho_{ct}^{1+}+\rho_{ct}^{2+})>0$.
  \item The network is sign ergodic if and only if $\rho_{ct}^3=\rho_{ct}^4=0$, $\rho_{cv}^3\rho_{cv}^2=0$.
  \item The network is structurally ergodic if and only if $\rho_{ct}^3=\rho_{ct}^4=0$.
\end{enumerate}
By mixing the different conditions, we then immediately obtain the associated conditions for the antithetic integral control of the reaction network.}

\section*{Acknowledgments}

This research has received funding from the European Research Council under the European Union's Horizon 2020 research and innovation programme / ERC grant agreement 743269 (CyberGenetics)

\end{document}